\documentclass[10pt, reqno]{amsart}
\usepackage{fancyhdr}
\usepackage{a4wide, amsthm, amscd, amsfonts, amssymb, graphicx,tikz, color, environ}
\usepackage{amsaddr}
\usepackage{amsmath,amssymb,amsfonts,amsthm, graphicx,color,fancyhdr}
\usepackage[latin1]{inputenc}
\usepackage{enumerate}

\usepackage{psfrag}
\usepackage[colorlinks=true, pdfstartview=FitV, allcolors=black]{hyperref}
\usepackage{psfrag,caption}
\usepackage[all]{xy}
\usepackage{comment}
\usepackage{appendix}
\usepackage{verbatim}
\usepackage{mathrsfs}
\usepackage{pdfsync}
\usepackage{graphicx,epsfig}
\newtheorem{thm}{Theorem}[section]

\newtheorem{lem}[thm]{Lemma}
\theoremstyle{stylename}

\let\oldproofname=\proofname
\renewcommand{\proofname}{\rm\bf{\oldproofname}}

\theoremstyle{definition}
 
 \newtheorem{rmk}[thm]{Remark}

\begin{document}

\title[ Sharp bounds for Steklov eigenvalues]
 {Some sharp bounds for Steklov eigenvalues}
\author{Sheela Verma}
\address{Tata Institute of Fundamental Research \\ Centre For Applicable Mathematics \\ Bangalore, India}
\email{sheela.verma23@gmail.com}

\author{G. Santhanam}
\address{Department of Mathematics and Statistics\\
Indian Institute of Technology Kanpur\\
Kanpur, India}
\email{santhana@iitk.ac.in}


\begin{abstract}
This work is an extension of a result given by Kuttler and Sigillito (SIAM Rev $10$:$368-370$, $1968$) on a star-shaped bounded domain in $\mathbb{R}^2$. Let $\Omega$ be a star-shaped bounded domain in a hypersurface of revolution, having smooth boundary. In this article, we obtain a sharp lower bound for all Steklov eigenvalues on $\Omega$ in terms of the Steklov eigenvalues of the largest geodesic ball contained in $\Omega$ with the same center as $\Omega$. We also obtain similar bounds for all Steklov eigenvalues on star-shaped bounded domain in paraboloid, $P = \left\lbrace (x, y, z) \in \mathbb{R}^{3} : z = x^2 + y^2\right\rbrace$.
\end{abstract}
\maketitle
\section{Introduction}
Let $\Omega $ be a bounded domain in a compact connected Riemannian manifold with smooth boundary $\partial\Omega$. The Steklov eigenvalue problem is to find all real numbers $\mu$ for which there exists a nontrivial function $\varphi \in C^{2}(\Omega) \cap C^{1}(\overline{\Omega})$ such that 
\begin{align} \label{Steklov problem}
\begin{array}{rcll}
\Delta \varphi &=& 0 & \mbox{ in } \Omega ,\\
\frac{\partial \varphi}{\partial \nu} &=& \mu \varphi  &\mbox{ on } \partial \Omega,
\end{array}
\end{align}
where $\nu$ is the outward unit normal to the boundary $\partial \Omega$.
This problem was introduced by Steklov \cite{S} for bounded domains in the plane in $1902$. Its importance lies in the fact that the set of eigenvalues of the Steklov problem is same as the set of eigenvalues of the well-known Dirichlet-Neumann map. This map associates to each function defined on $\partial \Omega$, the normal derivative of its harmonic extension on $\Omega$. The eigenvalues of the Steklov problem are discrete and form an increasing sequence $0= \mu_1 < \mu_2 \leq \mu_3 \leq \cdots \nearrow \infty$. The variational characterization of $\mu_l$, $1 \leq l < \infty$ is given by
\begin{align} \label{general variational characterization}
\mu_{l}(\Omega)= \sup_{E} \inf_{0 \neq \varphi \in E^{\perp}} \frac{\int_\Omega{\|\nabla \varphi\|^2}\, dv}{\int_{\partial\Omega}{\varphi^2}\, ds},
\end{align}
where $E$ is a set of ${l-1}$ functions $\phi_{1}, \phi_{2}, \ldots, \phi_{l-1}$ such that $\phi_{i} \in H^{1}(\Omega)$, $1 \leq i \leq {l-1}$ and $ E^{\perp} = \left\lbrace \varphi \in H^{1}(\Omega) : \int_{\partial\Omega} \varphi \phi_{i} ds = 0, 1 \leq i \leq {l-1} \right\rbrace$. For background on this problem, see \cite{GP}.

There are several results which estimate first nonzero eigenvalue of the Steklov eigenvalue problem \cite{BG, BP, E2, E3}. The first upper bound for $\mu_{2}$ was given by Weinstock \cite{W} in $1954$. He proved that among all simply connected planar domains with analytic boundary of fixed perimeter, the circle maximizes $\mu_{2}$. Later F. Brock \cite{B} obtained a sharp upper bound for $\mu_{2}$ by fixing the volume of the domain. He proved that for a bounded Lipschitz domain $\Omega \subset \mathbb{R}^{n}$, $\mu_{2}(\Omega)\, \left( \text{vol}(\Omega)\right)^{\frac{1}{n}} \leq {\omega_{n}}^{\frac{1}{n}}$, where $\omega_{n}$ is the volume of the unit ball in $\mathbb{R}^{n}$ and equality holds if and only if $\Omega$ is a ball. In several recent papers, bounds for all eigenvalues of the Steklov problem have been studied \cite{CGG, GLS, PS, X}. In particular, sharp upper bounds for some specific functions of the Steklov eigenvalues have been derived in \cite{GLS}. Weyl-type bounds have also been obtained for Steklov eigenvalues in \cite{PS, X}.

Let $\Omega\subset \mathbb{R}^{n}$ be a star-shaped domain with smooth boundary $\partial\Omega$. Let $p$ be a center of $\Omega$. Let $R_m = \min\left\lbrace  d(p, x) | x \in \partial\Omega \right\rbrace $, $R_M = \max\left\lbrace d(p, x) | x \in \partial\Omega\right\rbrace $ and $h_m = \min\left\lbrace \langle x, \nu \rangle | x \in \partial\Omega \right\rbrace $, where $\nu$ is the outward unit normal to $\partial\Omega$. With these notations, Bramble and Payne \cite{BP} proved that
$$ \mu_{2}(\Omega) \geq \frac{{R_m}^{n-1}}{{R_M}^{n+1}} \, h_m. $$
Equality holds when $\Omega$ is a ball. 

Kuttler and Sigillito \cite{KS} proved the following lower bound for a star-shaped bounded domain in $\mathbb{R}^2$.
\begin{thm}[\cite{KS}]
Let $\Omega$ be a star-shaped bounded domain in $\mathbb{R}^2$ with smooth boundary and centered at the origin. Then, for $1 \leq k < \infty$, 
\begin{align*}
\mu_{2k+1}(\Omega)\geq \mu_{2k}(\Omega) \geq \frac{k \left[ 1- {2} \bigg/ \left(1+\sqrt{1+ 4\, \min \left( {R(\theta)} / {R'(\theta)}\right)^2}\right)\right]} {\max \sqrt{R^2(\theta) + {R'}^2(\theta)}}, 
\end{align*}
where $R(\theta) = \max \left\lbrace |x| : x \in \Omega, x = |x| e^{i\theta} \right\rbrace $  and equality holds for a disc. 
\end{thm}

Following the idea of Kuttler and Sigillito \cite{KS}, Garcia and Montano \cite{GM} and the first author \cite{SV} obtained a similar bound for the first nonzero Steklov eigenvalue on a star-shaped domain in $\mathbb{R}^n$ and $\mathbb{S}^n$, respectively. Let $\Omega$ be a star-shaped bounded domain with smooth boundary $\partial\Omega$ centered at a point $p$ and $\nu$ be the outward unit normal to $\partial\Omega$. For any point $q \in \partial\Omega$, let $0 \leq \theta(q) \leq \alpha < \frac{\pi}{2}$, where $\cos (\theta(q)) = \langle \nu(q), \partial_r(q) \rangle$. Let $a= \tan^{2} \alpha$.

\begin{thm}[\cite{GM}] \label{Thm:GM}
Let $\Omega \subset \mathbb{R}^n$. Then with the above notations, the first nonzero eigenvalue of the Steklov problem $\mu_{2}(\Omega)$ satisfies 
\begin{align*}
\mu_2(\Omega) \geq \frac{(R_{m})^{n-2}}{(R_{M})^{n-1}} \frac{\left\lbrace 2+a-\sqrt{a^2+4\,a}\right\rbrace }{2\sqrt{a+1}}.
\end{align*}

\end{thm}

\begin{thm}[\cite{SV}]
Let $\Omega$ be a star-shaped bounded domain in $\mathbb{S}^{n}$ such that $\Omega \subset \mathbb{S}^{n}\backslash \left\lbrace -p \right\rbrace$. Then the first nonzero Steklov eigenvalue $\mu_{2}(\Omega)$ satisfies 
\begin{align*}
\mu_{2} (\Omega) &\geq \left(\frac{R_{m}}{R_{M}} \right) \left( \frac{(2+a) - \sqrt{a^{2}+4\,a}}{2 \, \sqrt{1+ a}}\right)  \frac{\sin^{n-1}\left(R_{m}\right)}{ \sin^{n-1}\left(R_{M}\right)} \, \mu_{2}\left( B\left( R_{m}\right)\right).
\end{align*}
\end{thm}
Here $R_{m}$ and $R_{M}$ are defined as above.

In Theorem \ref{thm: for generalised starshaped}, we obtain a lower bound similar to \cite{SV}, for all Steklov eigenvalues on a star-shaped domain $\Omega$ in hypersurface of revolution centered at pole. In Theorem \ref{thm: star shaped domain in paraboloid}, we prove a result for a star-shaped domain in a paraboloid in $\mathbb{R}^{3}$ analogous to the above. The main tool used to prove these results is the construction of suitable test function for the variational characterization of the corresponding eigenvalues.  
\section{Eigenvalues on hypersurface of revolution}
Let $M$ be a hypersurface of revolution with metric $g = dr^{2} + h^{2}(r) g_{\mathbb{S}^{n-1}}$, where $g_{\mathbb{S}^{n-1}}$ is the usual metric on $\mathbb{S}^{n-1}$ and $r \in \left[0, L \right] $ for some $L \in \mathbb{R}^{+}$. Moreover, We assume that $h$ satisfies $h(0) = 0$, $h'(0) = 1$. Let $\Omega$ be a star-shaped bounded domain in $M$ with respect to the pole $p$ of $M$. Let $\partial\Omega$ be the smooth boundary of $\Omega$ with outward unit normal $\nu$. Since $\Omega$ is star-shaped with respect to the point $p$ and have smooth boundary, then for every point $q \in \partial\Omega,$ there exists a unique unit vector $u \in T_{p}M$ and $R_{u} >0$ such that $q = \exp_{p} (R_{u} \, u)$. Observe that in geodesic polar coordinates, $\Omega$ and $\partial\Omega$ can be written as
\begin{align*}
\partial\Omega &= \left\lbrace ( R_{u}, u)  : u \in T_{p} \mathbb{S}^{n}, \Vert u \Vert=1\right\rbrace  \mbox{ and } \\
\Omega\backslash \left\lbrace p\right\rbrace &= \left\lbrace (r, u)  : u \in T_{p} \mathbb{S}^{n}, \Vert u \Vert=1,  0 < r < R_{u}\right\rbrace.
\end{align*}
Define $R_{m}=\min {R_{u}}, \, R_{M}=\max {R_{u}}$. 

Let $\partial_r$ be the radial vector field starting at $p$, the center of $\Omega$ and $\nu$ be the unit outward normal to $\partial\Omega$. Since $\Omega$ is a star-shaped bounded domain, for any point $q \in \partial\Omega$,  $\cos (\theta(q)) = \langle \nu(q), \partial_r(q) \rangle > 0 $. Therefore $\theta(q) < \frac{\pi}{2}$ for all $q \in \partial\Omega$. By compactness of $\partial\Omega$, there exists a constant $\alpha$ such that $0 \leq \theta(q) \leq \alpha  < \frac{\pi}{2}$ for all $q \in \partial\Omega$. Recall that for any point $q \in \partial\Omega$,  $\tan^{2} (\theta(q))= \frac{\|\overline{\nabla} R_{u}\|^2}{h^{2}(R_{u})}$. Additionally, assume that $h$ also satisfies the following conditions 
\begin{enumerate}[(a)]
\item $\frac{h(r)}{r}$ is a decreasing function of $r$ on $\left[0, R_{M}\right] $,
\item $h(r)$ is an increasing function of $r$ on $\left[0, R_{M}\right] $.
\end{enumerate}

\begin{lem}
Let $h(r)$ be a function defined on $[0, R]$ such that $\frac{h(r)}{r}$ is a decreasing function. Then $h(r)$ satisfies the following properties:
\begin{enumerate}[(a)]
\item If $0 \leq a \leq 1$, then $h(ar) \geq a h(r)$. 
\item If $a \geq 1$, then $h(ar) \leq a h(r)$.
\end{enumerate}
\end{lem}
\begin{proof} Since $\frac{h(r)}{r}$ is a decreasing function of $r$,  
\begin{align*}
\text{ for }  0 \leq a \leq 1, 0 \leq ar \leq r \text{ and } \frac{h(r)}{r} \leq \frac{h(ar)}{ar}  \\
\text{ for }  a \geq 1, ar \geq r \text{ and } \frac{h(r)}{r} \geq \frac{h(ar)}{ar} .
\end{align*}
Which gives the desired results.
\end{proof}

The following theorem gives a sharp lower bound for all Steklov eigenvalues on a star-shaped domain in $M$.
\begin{thm} \label{thm: for generalised starshaped}
Let $\Omega \subset M $, $\nu$, $\alpha$, $R_{m}$ and $R_{M}$ be as the above. Let $a = \tan^{2}(\alpha) $. Then $\mu_{l}(\Omega)$, $1 \leq l < \infty$ satisfies the following inequality.
\begin{align} \label{main ineq: lower bound}
\mu_{l} (\Omega) &\geq \left(\frac{R_{m}}{R_{M}} \right) \left( \frac{(2+a) - \sqrt{a^{2}+4\,a}}{2 \, \sqrt{1+ a}}\right)  \frac{h^{n-1}\left(R_{m}\right)}{ h^{n-1}\left(R_{M}\right)} \, \mu_{l}\left( B\left( R_{m}\right)\right),
\end{align}
where $B\left( R_{m}\right)\subset M$ is the geodesic ball of radius $R_{m}$ centered at $p$. Further, if $\Omega$ is a geodesic ball, then equality occurs. Conversely, if equality holds for some $l$, then $\Omega$ is a geodesic ball of radius $R_m$.
\end{thm}

\begin{proof}
For a continuously differential real valued function $f$ defined on $\overline{\Omega}$, we first find a lower bound for $ \int_\Omega{\|\nabla f\|^2}\, dv $ and then an upper bound for $  \int_{\partial\Omega}{f^2}\, ds $ to find a lower bound for  $ \frac{\int_\Omega{\|\nabla f\|^2}\, dv}{\int_{\partial\Omega}{f^2}\, ds} $.

Let $f$ be a continuously differential real valued function defined on $\overline{\Omega}$. Then for $q \in \Omega$, $\|\nabla f\|^2 = \left(\frac{\partial f}{\partial r} \right)^2 + \frac{1}{h^{2}(r) }\|\overline{\nabla} f\|^2$.
Therefore
\begin{align*}
\int_\Omega{\|\nabla f\|^2}\, dv = \int_{U_{p}\Omega} \int_{0}^{R_{u}} \left[ \left(\frac{\partial f}{\partial r} \right)^2 + \frac{1}{h^{2}(r) }\|\overline{\nabla} f\|^2\right] h^{n-1}(r) \, dr\, du.
\end{align*}
Let $u' = u$, $\rho = \frac{r \, R_{m}}{R_{u}}$. Then $\overline{\nabla} f = \overline{\nabla}_{u'} f - \frac{\rho}{R_{u'}} \frac{\partial f} {\partial \rho} \, \overline{\nabla}_{u'} R_{u'}$. By abuse of notations, we denote $u'$ by $u$ and $\overline{\nabla}_{u'}$ by $\overline{\nabla}$. Then the above integral can be written as
\begin{align*}
\int_\Omega{\|\nabla f\|^2}\, dv &= \int_{U_{p}\Omega} \int_{0}^{R_{m}} \left[ \left(\frac{R_{m}}{R_{u}} \right)^{2} \left(\frac{\partial f}{\partial \rho} \right)^2 + \frac{1}{h^{2}\left(\frac{\rho \, R_{u}}{R_{m}} \right) }\left\lbrace \|\overline{\nabla} R_{u}\|^2\left(\frac{\rho}{R_{u}} \frac{\partial f}{\partial \rho} \right)^2 \right. \right.\\
& \qquad \left. \left.+ \|\overline{\nabla} f\|^2 - \frac{2\, \rho}{R_{u}} \frac{\partial f}{\partial \rho}\, \langle\overline{\nabla} f, \overline{\nabla} R_{u}\rangle \right\rbrace \right] h^{n-1}\left(\frac{\rho \, R_{u}}{R_{m}} \right) \, \left( \frac{R_{u}}{R_{m}}\right) \, d\rho\, du.
\end{align*}
Next we estimate $\langle\overline{\nabla} f, \overline{\nabla} R_{u}\rangle$. For any function $\beta^2$ on $\overline{\Omega}$, Cauchy-Schwarz inequality gives
\begin{align*}
- \frac{2\, \rho}{R_{m}\, h^{2}\left(\frac{\rho \, R_{u}}{R_{m}} \right)} \left( \frac{\partial f}{\partial \rho}\right) \, \langle\overline{\nabla} f, \overline{\nabla} R_{u}\rangle & \geq - \frac{1}{\beta^2}\frac{\|\overline{\nabla} R_{u}\|^2}{R_{u}\,R_{m} }\left(\frac{\rho}{h \left(\frac{\rho \, R_{u}}{R_{m}} \right) } \right)^2 \left(\frac{\partial f}{\partial \rho} \right)^2 \\
& \qquad -\frac{\beta^2 \, R_{u}}{R_{m}\,h^{2}\left(\frac{\rho \, R_{u}}{R_{m}} \right) }\|\overline{\nabla} f\|^2.
\end{align*}
Thus
\begin{align} \nonumber
\int_\Omega{\|\nabla f\|^2}\, dv & \geq \int_{U_{p}\Omega} \int_{0}^{R_{m}} \left[\left\lbrace \left(\frac{R_{m}}{R_{u}} \right) - \left(\frac{1}{\beta^2} - 1\right)  \frac{\|\overline{\nabla} R_{u}\|^2}{R\,R_{m} }\left(\frac{\rho}{h \left(\frac{\rho \, R_{u}}{R_{m}} \right) } \right)^2 \right\rbrace  \left(\frac{\partial f}{\partial \rho} \right)^2 \right. \\\label{grad f expression}
& \qquad \left. + \frac{R_{u}\, \left( 1- \beta^2 \right) }{R_{m}\,h^{2}\left(\frac{\rho \, R_{u}}{R_{m}} \right) }\|\overline{\nabla} f\|^2  \right] h^{n-1}\left(\frac{\rho \, R_{u}}{R_{m}} \right) \, d\rho\, du. 
\end{align}
Note that $0 \leq \frac{\rho}{R_{m}} \leq 1 \leq \frac{R_{u}}{R_{m}}$ and $0 \leq \rho \leq \frac{\rho\, R_{u}}{R_{m}} \leq R_{u}$. Hence
\begin{align}\label{h inequality1}
\begin{split}
\frac{\rho}{R_{m}} h(R_{u})  \leq h\left( \frac{\rho\, R_{u}}{R_{m}}\right) \leq \frac{R_{u}}{R_{m}} h(\rho) ,\\
0 \leq h^{n-1}(\rho) \leq h^{n-1}\left( \frac{\rho\, R_{u}}{R_{m}}\right) .
\end{split}
\end{align} 
We assume $\beta^{2}<1$ and by substituting above inequalities in \eqref{grad f expression}, we get
\begin{align*} \nonumber
\int_\Omega{\|\nabla f\|^2}\, dv & \geq \int_{U_{p}\Omega} \int_{0}^{R_{m}} \left[\left\lbrace \left(\frac{R_{m}}{R_{u}} \right) - \left(\frac{1}{\beta^2} - 1 \right)  \frac{\|\overline{\nabla} R_{u}\|^2}{R_{u}\,R_{m} }\left(\frac{R_{m}}{h(R_{u})} \right)^2 \right\rbrace  \left(\frac{\partial f}{\partial \rho} \right)^2 \right. \\\nonumber
& \qquad \left. + \frac{R_{u}\, \left( 1- \beta^2 \right) }{R_{m}} \left(\frac{R_{m}}{R_{u}\,h(\rho)} \right)^2  \|\overline{\nabla} f\|^2  \right] h^{n-1}\left(\rho\right) \, d\rho\, du \\\nonumber
& \geq \left(\frac{R_{m}}{R_{M}} \right)\int_{U_{p}\Omega} \int_{0}^{R_{m}} \left[\left\lbrace 1 - \left(\frac{1}{\beta^2} - 1\right)  a \, \right\rbrace  \left(\frac{\partial f}{\partial \rho} \right)^2 \right. \\
& \qquad \left. + \frac{ \left( 1- \beta^2 \right) }{h^{2}(\rho)} \|\overline{\nabla} f\|^2  \right] h^{n-1}\left(\rho\right) \, d\rho\, du. 
\end{align*}
By solving the equation $ 1 - \left(\frac{1}{\beta^2} - 1\right) a =  1- \beta^2 $ for $\beta^2$ we see that
$$1 - \left(\frac{1}{\beta^2} - 1\right) a =1- \beta^{2} = \frac{(2+a) - \sqrt{a^{2}+4\,a}}{2} > 0. $$
From this it follows that
\begin{align} \nonumber
\int_\Omega{\|\nabla f\|^2}\, dv & \geq \left(\frac{R_{m}}{R_{M}} \right) \left( \frac{(2+a) - \sqrt{a^{2}+4\,a}}{2}\right)  \int_{U_{p}\Omega} \int_{0}^{R_{m}} \left[\left(\frac{\partial f}{\partial \rho} \right)^2 \right. \\ \nonumber
& \qquad \left. + \frac{1}{h^{2}(\rho)} \|\overline{\nabla} f\|^2  \right] h^{n-1}\left(\rho\right) \, d\rho\, du\\ \label{grad f inequality}
&= \left(\frac{R_{m}}{R_{M}} \right) \left( \frac{(2+a) - \sqrt{a^{2}+4\,a}}{2}\right)  \int_{{B(R_{m})}}  {\|\nabla f\|^2}\, dv.
\end{align} 

Now we find an upper bound for $\int_{\partial\Omega}{f^2}\, ds$.

Recall that the Riemannian volume element on $\partial\Omega$, denoted $ds$, is given by $ds = \sec (\theta)\, h^{n-1}\left(R_{u}\right)\, du$ (see \cite{TF}). Then
\begin{align*}
\int_{\partial\Omega}{f^2}\, ds = \int_{U_{p}\Omega} f^2 \, \sec (\theta)\, h^{n-1}\left(R_{u}\right)\, du. 
\end{align*}
By using the fact that $h^{n-1} (R_{m}) \leq h^{n-1} (R_{u}) \leq h^{n-1} (R_{M}) $ and substituting $r = \frac{\rho \, R_{u}}{R_{m}}$, this integral becomes
\begin{align} \label{f^2 inequality}
\int_{\partial\Omega}{f^2}\, ds &\leq \frac{\sec (\alpha) \, h^{n-1}\left(R_{M}\right)}{h^{n-1}\left(R_{m}\right)}\int_{S(R_{m})} f^2 \, ds.
\end{align}
By inequalities \eqref{grad f inequality} and \eqref{f^2 inequality}, we have
\begin{align} \label{inequality for vari char for f}
\frac{\int_\Omega{\|\nabla f\|^2}\, dv}{\int_{\partial\Omega}{f^2}\, ds} &\geq \left(\frac{R_{m}}{R_{M}} \right) \left( \frac{(2+a) - \sqrt{a^{2}+4\,a}}{2}\right)  \frac{h^{n-1}\left(R_{m}\right)}{\sec (\alpha) \, h^{n-1}\left(R_{M}\right)}\frac{\int_{{B(R_{m})}}  {\|\nabla f\|^2}\, dv}{\int_{S(R_{m})} f^2 \, ds} .
\end{align}
We now construct some specific test functions for the variational characterization of $\mu_{l} (\Omega)$.

We choose the functions $\phi_{i}$, $1 \leq i < \infty$ such that $\phi_{i} h^{n-2}(R_{u}) \sqrt{h^{2}(R_{u})+ \|\overline{\nabla} R_{u}\|^2}$ is the $i$th Steklov eigenfunction of $B(R_{m})$. Let $\varphi$ be an arbitrary function which satisfies
\begin{align*}
\int_{\partial B(R_{m})} \varphi \phi_{i} h^{n-2}(R_{u}) \sqrt{h^{2}(R_{u})+ \|\overline{\nabla} R_{u}\|^2}\, ds =0.
\end{align*}
Note that
\begin{align*}
\int_{\partial \Omega} \varphi \phi_{i} ds = \int_{U_{p}\Omega} \varphi \phi_{i} \, \frac{\sqrt{h^{2}(R_{u})+ \|\overline{\nabla} R_{u}\|^2}}{h(R_{u})}\, h^{n-1}\left(R_{u}\right)\, du.
\end{align*}
By substituting $r = \frac{\rho \, R_{u}}{R_{m}}$, the above integral becomes
\begin{align*}
\int_{\partial \Omega} \varphi \phi_{i} ds &= \frac{1}{h^{n-1}\left(R_{m}\right)}\int_{\partial B(R_{m})} \varphi \phi_{i} \, \sqrt{h^{2}(R_{u})+ \|\overline{\nabla} R_{u}\|^2}\, h^{n-2}\left(R_{u}\right)\, ds\\
&=0.
\end{align*}
Fix $E= \left\lbrace \phi_{1}, \phi_{2},\ldots, \phi_{l-1}\right\rbrace$ in \eqref{general variational characterization}. Then it follows from \eqref{general variational characterization} that
\begin{align} \nonumber
\mu_{l} (\Omega) &\geq \inf_{\substack{\varphi \neq 0 \\ \int_{\partial \Omega} \varphi \phi_{i} ds=0,\\ 1 \leq i \leq l-1}} \frac{\int_\Omega{\|\nabla \varphi\|^2}\, dv}{\int_{\partial\Omega}{\varphi^2}\, ds}\\ \nonumber
& \geq \left(\frac{R_{m}}{R_{M}} \right) \left( \frac{(2+a) - \sqrt{a^{2}+4\,a}}{2\, \sqrt{1+a}}\right)  \frac{h^{n-1}\left(R_{m}\right)}{h^{n-1}\left(R_{M}\right)} \\ \label{prefinal inequality}
& \qquad \inf_{\substack{\varphi \neq 0 \\ \int_{\partial B(R_{m})} \varphi \phi_{i} h^{n-2}(R_{u}) \sqrt{h^{2}(R_{u})+ \|\overline{\nabla} R_{u}\|^2} ds=0,\\ 1 \leq i \leq l-1}} \frac{\int_{{B(R_{m})}}  {\|\nabla \varphi\|^2}\, dv}{\int_{\partial B(R_{m})} \varphi^2 \, ds}.
\end{align}
Since $\phi_{i} h^{n-2}(R_{u}) \sqrt{h^{2}(R_{u})+ \|\overline{\nabla} R_{u}\|^2}$ is the $i$th Steklov eigenfunction of $B(R_{m})$, we have
\begin{align*}
\inf_{\substack{0 \neq \varphi \\ \int_{\partial B(R_{m})} \varphi \phi_{i} h^{n-2}(R_{u}) \sqrt{h^{2}(R_{u})+ \|\overline{\nabla} R_{u}\|^2} ds=0,\\ 1 \leq i \leq l-1}} \frac{\int_{{B(R_{m})}}  {\|\nabla \varphi\|^2}\, dv}{\int_{\partial B(R_{m})} \varphi^2 \, ds} = \mu_{l}\left( B\left( R_{m}\right)\right).
\end{align*}
By substituting the above value in \eqref{prefinal inequality}, we get \eqref{main ineq: lower bound}. If $\Omega$ is a geodesic ball, then $R_{m}= R_{M} $ and $a = 0$, hence equality holds in \eqref{main ineq: lower bound}. Next if equality holds in \eqref{main ineq: lower bound} for some $l$, then equality holds in \eqref{h inequality1} and $R_{u}= R_{m}$. Hence $\Omega$ is a geodesic ball.
\end{proof} 

\begin{rmk}
In \cite{GM} and \cite{SV}, authors obtained a lower bound for the first nonzero Steklov eigenvalue on a star-shaped bounded domain in $\mathbb{R}^{n}$ and $\mathbb{S}^{n}$, respectively. Using the above idea, a similar bound can be obtained for all nonzero Steklov eigenvalues on a star-shaped bounded domain in $\mathbb{R}^{n}$ and $\mathbb{S}^{n}$.
\end{rmk}

\section{Eigenvalues on a paraboloid in $\mathbb{R}^{3}$}
In this section, we state and prove the result for a star-shaped bounded domain in a paraboloid $P = \left\lbrace (x, y, z) \in \mathbb{R}^{3} : z = x^2 + y^2\right\rbrace$. We first fix some notations which will be used to state the main result of this section.

We use the parametrization $\left( r \cos\theta, r \sin\theta, r^{2}\right)$ for paraboloid $P$, where $ \theta \in [0, 2 \pi)$ and $r \geq 0$. Then the line element $ds^{2}$ and the area element $dA$ on $P$ is given by $ds^{2}= \left( 1+ 4r^{2} \right)dr^{2} + r^{2} \, d\theta^{2}$ and $dA = r \sqrt{1+ 4r^{2}}\, dr \, d\theta$, respectively. Let $\Omega \subset P$ be a star-shaped bounded domain with respect to the origin and have smooth boundary $\partial \Omega$. Then there exists a function $R:[0, 2 \pi)\longrightarrow \mathbb{R}^{+}$ such that 
\begin{align*}
\partial\Omega &= \left\lbrace ( R(\theta), \theta)  : \theta \in [0, 2 \pi)\right\rbrace  \mbox{ and } \\
\Omega\backslash \left\lbrace 0\right\rbrace &= \left\lbrace (r, \theta)  : \theta \in [0, 2 \pi), 0 < r < R(\theta)\right\rbrace.
\end{align*}
Hereafter, we denote $R(\theta)$ by $R_{\theta}$. Let $R_m = \text{min}\left\lbrace R_{\theta}: \theta \in [0, 2 \pi)\right\rbrace $ and $R_M = \text{max}\left\lbrace R_{\theta}: \theta \in [0, 2 \pi)\right\rbrace $. Define $B(R_{m}) = \left\lbrace ( R_{m}, \theta)  : \theta \in [0, 2 \pi)\right\rbrace$. Let $\nu$ be the outward unit normal to $\partial \Omega$. Let $a= \max\left\lbrace \left( 1+ 4 R_{\theta}^{2}\right)\left( \frac{R_{\theta}'}{R_{\theta}}\right)^{2} : \theta \in [0, 2 \pi)  \right\rbrace $. With these notations, we prove the following theorem.
\begin{thm} \label{thm: star shaped domain in paraboloid}
Let $\Omega$, $\nu$, $a$, $R_{m}$ and $R_{M}$ be as the above. Then $\mu_{l}(\Omega)$, $1 \leq l < \infty$ satisfies 
\begin{align}  \label{main ineq: lower bound in paraboloid}
\mu_{l} (\Omega) &\geq \left(\frac{R_{m}}{R_{M}} \right)^{3} \left( \frac{(2+a) - \sqrt{a^{2}+4\,a}}{2 \sqrt{1+a}}\right) \mu_{l}\left( B\left( R_{m}\right)\right).
\end{align}
Furthermore, if equality holds for some $l$ then $\Omega$ is a geodesic ball of radius $R_{m}$ and if $\Omega$ is a geodesic ball then equality holds in \eqref{main ineq: lower bound in paraboloid}.
\end{thm}
\begin{proof}
Let $f$ be a continuously differentiable real valued function defined on $\overline{\Omega}$. We first obtain a lower bound for $\int_\Omega{\|\nabla f\|^2}\, dA$.
\begin{align*}
\int_\Omega{\|\nabla f\|^2}\, dA &= \int_{0}^{2 \pi} \int_{0}^{R_{\theta}} \left[\frac{1}{1 + 4 r^{2}} \left(\frac{\partial f}{\partial r} \right)^2 + \frac{1}{r^{2}}\left(\frac{\partial f}{\partial \theta} \right)^2\right] r \sqrt{1 + 4 r^{2}} \, dr\, d\theta \\
&= \int_{0}^{2 \pi} \int_{0}^{R_{\theta}} \left[\frac{r}{\sqrt{1 + 4 r^{2}}} \left(\frac{\partial f}{\partial r} \right)^2 +\frac{\sqrt{1 + 4 r^{2}}}{r} \left( \frac{\partial f}{\partial \theta} \right)^2\right] dr\, d\theta
\end{align*}
Let $\phi = \theta$, $\rho = \frac{r \, R_{m}}{R_{\theta}}$. Since $\rho = \frac{r \, R_{m}}{R_{\theta}} \leq r$, we have $\sqrt{1 + 4 r^{2}} \geq \sqrt{1 + 4 \rho^{2}}$ and $\frac{r}{\sqrt{1 + 4 r^{2}}} \geq \frac{\rho}{\sqrt{1 + 4 \rho^{2}}} $. Thus the above integral can be written as
\begin{align*}
\int_\Omega{\|\nabla f\|^2}\, dA &\geq \int_{0}^{2 \pi} \int_{0}^{R_{m}} \left[\frac{\rho}{\sqrt{1 + 4 \rho^{2}}} \left( \frac{R_{m}}{R_{\phi}} \frac{\partial f}{\partial \rho} \right)^2 +\frac{R_{m} \sqrt{1 + 4 \rho^{2}}}{\rho R_{\phi}} \left( \frac{\partial f}{\partial \phi} - \frac{\rho R_{\phi}'}{R_{\phi}} \frac{\partial f}{\partial \rho} \right)^2\right] \\
& \qquad \frac{R_{\phi}}{R_{m}} \, d\rho\, d\phi \\
&= \int_{0}^{2 \pi} \int_{0}^{R_{m}} \left[\frac{\rho}{\sqrt{1 + 4 \rho^{2}}} \left( \frac{\partial f}{\partial \rho} \right)^2 + \frac{R_{\phi} \sqrt{1 + 4 \rho^{2}}}{\rho R_{m}} \left( \frac{\partial f}{\partial \phi} - \frac{\rho R_{\phi}'}{R_{\phi}} \frac{\partial f}{\partial \rho} \right)^2\right] \\
& \qquad    \frac{R_{m}}{R_{\phi}} d\rho\, d\phi \\
& \geq \int_{0}^{2 \pi} \int_{0}^{R_{m}} \left[\frac{\rho}{\sqrt{1 + 4 \rho^{2}}} \left( \frac{\partial f}{\partial \rho} \right)^2 + \frac{ \sqrt{1 + 4 \rho^{2}}}{\rho} \left\lbrace  \left( \frac{\partial f}{\partial \phi}\right)^{2} + \left(\frac{\rho R_{\phi}'}{R_{\phi}}\frac{\partial f}{\partial \rho} \right)^{2} \right. \right. \\
& \qquad \left. \left. - 2  \frac{\rho R_{\phi}'}{R_{\phi}} \frac{\partial f}{\partial \rho} \frac{\partial f}{\partial \phi} \right\rbrace \right]  \frac{R_{m}}{R_{\phi}} d\rho\, d\phi. 
\end{align*}
For any function $\beta^2$ on $\overline{\Omega}$, Cauchy-Schwarz inequality gives
\begin{align*}
- 2 \frac{\rho R_{\phi}'}{R_{\phi}} \frac{\partial f}{\partial \rho} \frac{\partial f}{\partial \phi} \geq - \frac{1}{\beta^2}\left(\frac{\rho R_{\phi}'}{R_{\phi}} \right)^2 \left(\frac{\partial f}{\partial \rho} \right)^2  -\beta^2 \left(\frac{\partial f}{\partial \phi} \right)^2.
\end{align*}
As a consequence, we have
\begin{align*}
\int_\Omega{\|\nabla f\|^2}\, dA & \geq \int_{0}^{2 \pi} \int_{0}^{R_{m}} \left[\frac{\rho}{\sqrt{1 + 4 \rho^{2}}} \left( \frac{\partial f}{\partial \rho} \right)^2 + \frac{ \sqrt{1 + 4 \rho^{2}}}{\rho} \left\lbrace \left(1 - \beta^{2} \right)  \left( \frac{\partial f}{\partial \phi}\right)^{2}  \right. \right. \\
& \qquad \left. \left. -\left(\frac{1}{\beta^2} - 1 \right) \left(\frac{\rho R_{\phi}'}{R_{\phi}}\frac{\partial f}{\partial \rho} \right)^{2} \right\rbrace \right]  \frac{R_{m}}{R_{\phi}} d\rho\, d\phi \\
&= \int_{0}^{2 \pi} \int_{0}^{R_{m}} \left[\left\lbrace 1  - \left(1 + 4 \rho^{2}\right) \left( \frac{1}{\beta^2} -1 \right) \left( \frac{ R_{\phi}'}{R_{\phi}}\right)^{2}\right\rbrace  \frac{\rho}{\sqrt{1 + 4 \rho^{2}}} \left(\frac{\partial f}{\partial \rho} \right)^2 \right.  \\
& \qquad \left.+ \left(1 - \beta^2 \right) \frac{ \sqrt{1 + 4 \rho^{2}}}{\rho} \left(\frac{\partial f}{\partial \phi} \right)^2 \right]   \frac{R_{m}}{R_{\phi}} d\rho\, d\phi. 
\end{align*}

Note that $\left(1 + 4 \rho^{2}\right) \left( \frac{ R_{\phi}'}{R_{\phi}}\right)^{2} \leq \left(1 + 4 R_{\phi}^{2}\right) \left( \frac{ R_{\phi}'}{R_{\phi}}\right)^{2} \leq a $ and $\frac{R_{m}}{R_{\phi}} \geq \frac{R_{m}}{R_{M}}$. Let's assume $\beta^{2}<1$, then the above integral becomes
\begin{align*} 
\int_\Omega{\|\nabla f\|^2}\, dA & \geq \left( \frac{R_{m}}{R_{M}}\right)\int_{0}^{2 \pi} \int_{0}^{R_{m}} \left[\left\lbrace 1 -  \left( \frac{1}{\beta^2} -1 \right) a \right\rbrace  \frac{\rho}{\sqrt{1 + 4 \rho^{2}}} \left(\frac{\partial f}{\partial \rho} \right)^2 \right.  \\
& \qquad \left.+ \left(1 - \beta^2 \right)\frac{ \sqrt{1 + 4 \rho^{2}}}{\rho}   \left(\frac{\partial f}{\partial \phi} \right)^2 \right]  d\rho\, d\phi. 
\end{align*} 
Solving the equation $1 -  \left( \frac{1}{\beta^2} -1 \right) a = 1 - \beta^2$ for $\beta^{2}$, we obtain
\begin{align*}
1 -  \left( \frac{1}{\beta^2} -1 \right) a = 1 - \beta^2 = \frac{(2+a) - \sqrt{a^{2}+4\,a}}{2} > 0.
\end{align*} 
By substituting these values, we have
\begin{align} \nonumber
\int_\Omega{\|\nabla f\|^2}\, dA & \geq \left( \frac{R_{m}}{R_{M}}\right)\frac{(2+a) - \sqrt{a^{2}+4\,a}}{2}\int_{0}^{2 \pi} \int_{0}^{R_{m}} \left[ \frac{\rho}{\sqrt{1 + 4 \rho^{2}}} \left(\frac{\partial f}{\partial \rho} \right)^2  \right. \\ \nonumber
& \qquad \left.+ \frac{ \sqrt{1 + 4 \rho^{2}}}{\rho} \left(\frac{\partial f}{\partial \phi} \right)^2 \right] d\rho\, d\phi \\ \nonumber
&= \left( \frac{R_{m}}{R_{M}}\right)\frac{(2+a) - \sqrt{a^{2}+4\,a}}{2}\int_{0}^{2 \pi} \int_{0}^{R_{m}} \left[ \frac{1}{1 + 4 \rho^{2}} \left(\frac{\partial f}{\partial \rho} \right)^2 + \frac{1}{\rho^{2}}  \right. \\ \nonumber
& \qquad \left.  \left(\frac{\partial f}{\partial \phi} \right)^2 \right]\rho \sqrt{1 + 4 \rho^{2}} \, d\rho\, d\phi \\ \label{gradf expre parabo.}
&= \left( \frac{R_{m}}{R_{M}}\right)\frac{(2+a) - \sqrt{a^{2}+4\,a}}{2}\int_{B(R_{m})} {\|\nabla f\|^2}\, dA.
\end{align}

Now we give a lower bound for $\int_{\partial\Omega}{f^2}\, ds$.
\begin{align*}
\int_{\partial\Omega}{f^2}\, ds &= \int_{0}^{2 \pi} f^2 \sqrt{1 + \left( 1+ 4 R_{\theta}^{2}\right)\left( \frac{R_{\theta}'}{R_{\theta}}\right)^{2}} R_{\theta} \, d\theta \\
&\leq \sqrt{1+a} \int_{0}^{2 \pi} f^2 \, R_{\theta} \, d\theta.
\end{align*}
By substituting $\phi = \theta$, $\rho = \frac{r \, R_{m}}{R_{\theta}}$ and using the fact that $R_{\theta} \leq R_{M}$, we get
\begin{align} \label{f^2 parabo}
\int_{\partial\Omega}{f^2}\, ds \leq \frac{R_{M}\sqrt{1+a}}{R_{m}} \int_{0}^{2 \pi} f^2 \, R_{m} \, d\phi
&= \frac{R_{M}\sqrt{1+a}}{R_{m}} \int_{\partial B(R_{m})} f^2 \, ds.
\end{align}
Hence for a continuously differentiable real valued function $f$ defined on $\overline{\Omega}$, it follows from \eqref{gradf expre parabo.} and \eqref{f^2 parabo} that
\begin{align*} \label{expre of rayle on parabo}
\frac{\int_\Omega{\|\nabla f\|^2}\, dA}{\int_{\partial\Omega}{f^2}\, ds} \geq \left( \frac{R_{m}}{R_{M}}\right)^{2}\frac{(2+a) - \sqrt{a^{2}+4\,a}}{2\sqrt{1+a}} \, \frac{\int_{B(R_{m})} {\|\nabla f\|^2}\, dA}{\int_{\partial B(R_{m})} f^2 \, ds}.
\end{align*}
Now using the same argument as in Theorem \ref{thm: for generalised starshaped}, we get the desired result. 
\end{proof}

\subsection*{Acknowledgment}
The authors would like to thank Dr. Prosenjit Roy for various useful discussions. Some part of this work was done when the first author was working under project PDA/IITK/MATH/96062 at IIT Kanpur.

\end{document}